\documentclass[12pt]{amsart}

\usepackage{latexsym,rawfonts}
\usepackage{amsfonts,amssymb}
\usepackage{amsmath,amsthm}

\usepackage[plainpages=false]{hyperref}
\usepackage{graphicx}
\usepackage{color}

\newtheorem{theorem}{Theorem}
\newtheorem{corollary}{Corollary}

\newtheorem{lemma}{Lemma}

\theoremstyle{definition}

\DeclareMathOperator{\VOL}{vol}

\newcommand{\R}{\mathbb{R}}

\newcommand{\dd}{\mathop{}\!\mathrm{d}}

\newcommand{\set}[1]{\left\{#1\right\}}
\newcommand{\norm}[1]{\left\Vert#1\right\Vert}

\newcommand{\pd}{\partial}

\newcommand{\uS}{\mathbb{S}^{n-1}}

\newcommand{\MA}{Monge-Amp\`ere }
\newcommand{\OM}{Orlicz-Minkowski }

\newcommand{\beq}{\begin{equation}}
\newcommand{\eeq}{\end{equation}}
\newcommand{\beqs}{\begin{eqnarray*}}
\newcommand{\eeqs}{\end{eqnarray*}}
\newcommand{\beqn}{\begin{eqnarray}}
\newcommand{\eeqn}{\end{eqnarray}}

\begin{document}

\title{A generalized Gauss curvature flow related to the Orlicz-Minkowski problem}

\author{ YanNan Liu \qquad Jian Lu }

\address{YanNan Liu: School of Mathematics and Statistics, Beijing Technology and Business University, Beijing 100048, P.R. China}
\email{liuyn@th.btbu.edu.cn}

\address{Jian Lu: South China Research Center for Applied Mathematics and Interdisciplinary Studies, South China Normal University, Guangzhou 510631, P.R. China}
\email{jianlu@m.scnu.edu.cn}
\email{lj-tshu04@163.com}

\thanks{The authors were supported by Natural Science Foundation of China (11871432). The first author was also supported in part by Beijing Natural Science Foundation (1172005).}

\date{}

\begin{abstract}
In this paper a generalized Gauss curvature flow about a convex hypersurface in the
Euclidean $n$-space is studied. This flow is closely related to the
Orlicz-Minkowski problem, which involves Gauss curvature and a function of
support function.
Under some appropriate assumptions, we prove the long-time existence and
convergence of this flow.
As a byproduct, two existence results of solutions to the even Orlicz-Minkowski
problem are obtained, one of which improves the known result.
\end{abstract}



\maketitle
\vskip4ex

\section{Introduction}

Let $M_{0}$ be a smooth, closed, strictly convex hypersurface in the Euclidean
space $\R^{n}$, which encloses the origin and is given by a smooth embedding
$X_{0}: \mathbb{S}^{n-1} \rightarrow \R^{n}$. 
Consider a family of closed hypersurfaces $\set{M_t}$ with $M_t=X(\uS,t)$, where $X:
\mathbb{S}^{n-1}\times[0,T) \rightarrow \R^{n}$ is a smooth map satisfying the
following initial value problem: 
\begin{equation}\label{feq}
  \begin{split}
    \frac{\pd X}{\pd t} (x,t)
    &= -f(\nu)\mathcal{K}(x,t) \frac{\langle X, \nu \rangle}{\varphi(\langle X, \nu \rangle)} \eta(t) \nu +  X,\\
    X(x,0) &= X_{0}(x).
  \end{split}
\end{equation}
Here $f$ is a given positive smooth function on the unit sphere $\uS$, $\nu$ is the unit outer
normal vector of the hypersurface $M_t$ at the point $X(x,t)$, $\mathcal{K}(x,t)$ is the
Gauss curvature of $M_{t}$ at $X(x,t)$, $\langle \cdot,\cdot \rangle$ is the
standard inner product in $\R^n$, $\varphi$ is a positive smooth function
defined in $(0,+\infty)$, $\eta$ is a scalar function to be determined in
order to keep $M_t$ normalized in a certain sense, and $T$ is the maximal time
for which the solution exists.

When $f\equiv1$ and $\varphi(s) = s$, flow \eqref{feq} is a normalized Gauss
curvature flow.
The study of Gauss curvature flow was initiated by Firey \cite{Fir.M.21-1974.1}
for modeling shape change of tumbling stones.
Since then, various isotropic and anisotropic geometric flows involving Gauss
curvature have been extensively studied, see e.g.
\cite{And.IMRN.1997.1001,
AGN.Adv.299-2016.174,
GL.DMJ.75-1994.79,
Ger.CVPDE.49-2014.471,
Ham.CAG.2-1994.155,
Sch.JRAM.600-2006.117,
Sta.IMRNI.2012.2289,
Urb.JDG.33-1991.91}
and references therein.
For isotropic curvature flows, whether the limiting hypersurfaces are spheres or
not is an important issue, see results obtained in e.g.
\cite{And.Invent.138-1999.151, BCD.Acta.219-2017.1, Cho.JDG.22-1985.117}.
For anisotropic curvature flows, the limiting hypersurfaces are usually smooth
solutions to kinds of Minkowski type problems in convex geometry, providing
alternative methods of solving elliptic \MA type equations, see e.g.
\cite{BIS.AP.12-2019.259,
CHZ.MA.373-2019.953,
CL,
CW.AIHPANL.17-2000.733,
Iva.JFA.271-2016.2133, LL,
LSW.JEMSJ.22-2020.893}

The generalized anisotropic Gauss curvature flow \eqref{feq} is closely related
to the \OM problem arising in modern convex geometry.
In fact, by our main Theorem \ref{thm1} below, the support function $h$ of the limiting hypersurface of this
flow provides a smooth solution to the \MA type equation
\begin{equation} \label{OMP-f}
  c\, \varphi(h) \det(\nabla^2h +hI) =f \text{ on } \uS
\end{equation}
for some positive constant $c$.
Here $h$ is a function defined on $\uS$, $\nabla^2h =(\nabla_{ij}h)$ is the
Hessian matrix of covariant derivatives of $h$ with respect to an orthonormal
frame on $\uS$, and $I$ is the unit matrix of order $n-1$.
Equation \eqref{OMP-f} is just the smooth case of \OM problem.

The \OM problem is a basic problem in the Orlicz-Brunn-Minkowski theory in
convex geometry. This theory is the recent development of the classical
Brunn-Minkowski theory, and has attracted great attention from many scholars,
see for example 
\cite{
  GHW.JDG.97-2014.427, GHWY.JMAA.430-2015.810,
  HSX.MA.352-2012.517, HP.Adv.323-2018.114, HH.DCG.48-2012.281,
  HLYZ.Acta.216-2016.325,
  HZ.Adv.332-2018.57,
  HLYZ.DCG.33-2005.699,
  Lud.Adv.224-2010.2346,
  SL.Adv.281-2015.1364, XJL.Adv.260-2014.350,
  ZX.Adv.265-2014.132}
and references therein.
The \OM problem is a generalization of the classical Minkowski problem, and it
asks what are the necessary and sufficient conditions for
a Borel measure on the unit sphere $\uS$ to be a multiple of the Orlicz surface area
measure of a convex body in $\R^n$.
This problem is equivalent to solving equation \eqref{OMP-f} for some support
function $h$ and constant $c$ in smooth case.
When $\varphi(s)=s^{1-p}$, Eq. \eqref{OMP-f} reduces to the $L_p$-Minkowski
problem, which has been extensively studied, see e.g.
\cite{BHZ.IMRNI.2016.1807,
  BLYZ.JAMS.26-2013.831, 
  CW.Adv.205-2006.33,
  HLW.CVPDE.55-2016.117, HLX.Adv.281-2015.906, HLYZ.DCG.33-2005.699,
  JLW.Adv.281-2015.845, 
  JLZ.CVPDE.55-2016.41,
  Lu.SCM.61-2018.511, 
  LW.JDE.254-2013.983, Lut.JDG.38-1993.131,
  Zhu.Adv.262-2014.909, Zhu.IUMJ.66-2017.1333}
and Schneider's book \cite{Schneider.2014}, and corresponding references
therein.
For a general $\varphi$, several existence results have been known, see
\cite{BBC.AiAM.111-2019.101937,
  HLYZ.Adv.224-2010.2485,
  HH.DCG.48-2012.281,
  JL.Adv.344-2019.262}

In this paper we are concerned with the long-time existence and convergence of flow
\eqref{feq} for origin-symmetric convex hypersurfaces.
The special case when $\varphi(s)=s^{1-p}$ with $p>-n$ was first studied by Bryan,
Ivaki and Scheuer \cite{BIS.AP.12-2019.259}, and then by Sheng and Yi \cite{SY}
using a different flow.

In order to study the general case, we need to impose some constraints on
$\varphi$.
Two common assumptions are as follows:
\begin{quote}
{\bf (A)} $\varphi$ is a continuous and positive function defined in
$(0,+\infty)$ such that $\phi(s) = \int^{s}_{0} 1/\varphi(\tau)\dd\tau$ exists
for every $s>0$ and is unbounded as $s \rightarrow +\infty$; Or
\\
{\bf (B)} $\varphi$ is a continuous and positive function defined in
$(0,+\infty)$ such that for every $s>0$, $\phi(s) = \int^{+\infty}_{s}
1/\varphi(\tau)\dd\tau$ exists, and for $s$ near $0$, $\phi(s)\leq N s^p$ for
some positive constant $N$ and some number $p\in(-n,0)$.
\end{quote}
One can easily see that the special case $\varphi(s)=s^{1-p}$ satisfies
{\bf (A)} when $p>0$, and {\bf (B)} when $-n<p<0$.
In fact, these two assumptions were used in 
\cite{BBC.AiAM.111-2019.101937,
  HLYZ.Adv.224-2010.2485,
  HH.DCG.48-2012.281,
  JL.Adv.344-2019.262}
to prove existence results of equation \eqref{OMP-f} by variational methods.

As mentioned above, $\eta(t)$ in \eqref{feq} is used to keep $M_t$ normalized in
a certain sense. In this paper, we find that flow \eqref{feq} will evolve for a
long time if the volume of the convex body bounded by $M_t$ remains unchanged.
This requires $\eta$ to be given by
\begin{equation}\label{eta}
  \eta(t) = \frac{\int_{\uS}\rho(u,t)^n\dd u}{\int_{\uS} f(x)h(x,t)/\varphi(h) \dd x},
\end{equation}
where $\rho(\cdot,t)$ and $h(\cdot,t)$ are the radial function and support
function of the convex hypersurface $M_t$ respectively.
See section 2 for these definitions and computations.
Similar $\eta(t)$ was used by Chen, Huang and Zhao \cite{CHZ.MA.373-2019.953} to
study a geometric flow related to the $L_p$ dual Minkowski problem.

When $f$ is even, namely $f(-x)=f(x)$ for any $x\in\uS$, we obtain the following
long-time existence and convergence of flow \eqref{feq}.

\begin{theorem} \label{thm1}
  Assume $M_{0}$ is a smooth, closed, origin-symmetric, uniformly convex hypersurface in $\R^{n}$.
  If $f$ is a smooth and even function on $\uS$, and $\varphi\in
  C^\infty(0,+\infty)$ satisfies {\bf (A)} or {\bf (B)}, then flow \eqref{feq}
  has a unique smooth solution $M_{t}$ for all time $t > 0$.
  Moreover, when $t\to\infty$, a subsequence of $M_{t}$ converges in $C^{\infty}$ to a smooth,
  closed, origin-symmetric, uniformly convex hypersurface, whose support
  function is a smooth even solution to equation \eqref{OMP-f} for some positive
  constant $c$.
\end{theorem}

The study of flow \eqref{feq} is inspired by 
\cite{BIS.AP.12-2019.259,
CHZ.MA.373-2019.953,
CL, LSW.JEMSJ.22-2020.893}
where various Minkowski type problems were studied via different geometric
flows.
Our paper provides the first example of Gauss curvature flows related to the \OM problem.
As an application, we have
\begin{corollary} \label{corOM}
  Assume $f$ is a smooth and even function on $\uS$.
  If $\varphi$ is a smooth function satisfying {\bf (A)} or {\bf (B)}, then
  there exists a smooth even solution to equation \eqref{OMP-f} for some
  positive constant $c$.
\end{corollary}

The result of Corollary \ref{corOM} with assumption {\bf (A)} was obtained by
Haberl, Lutwak, Yang and Zhang \cite[Theorem 2]{HLYZ.Adv.224-2010.2485} for even measures.
The result of Corollary \ref{corOM} with assumption {\bf (B)} was obtained by
Bianchi, B\"{o}r\"{o}czky and Colesanti \cite{BBC.AiAM.111-2019.101937} for
$L_{\frac{n}{n+p}}$ functions (not necessarily even), where two more assumptions
on $\varphi$ were needed: $\lim_{s\to0^+}\varphi(s)=0$ and $\varphi$ is monotone increasing.

This paper is organized as follows.
In section 2, we give some basic knowledge about the flow \eqref{feq}.
In section 3, the long-time existence of flow \eqref{feq} will be proved.
First, under assumptions {\bf (A)} or {\bf (B)}, we derive uniform positive
upper and lower bounds for support functions of $\set{M_t}$.
Then, the bounds of principal curvatures are derived
via proper auxiliary functions and delicate computations.
So the long-time existence follows by standard arguments.
In section 4, by considering a related geometric functional, we prove that
a subsequence of $\set{M_t}$ converges to a smooth solution to equation
\eqref{OMP-f}, completing the proof of Theorem \ref{thm1}.

\section{Preliminaries}

Let $\R^n$ be the $n$-dimensional Euclidean space, and $\uS$ be the unit sphere
in $\R^n$.
Assume $M$ is a smooth closed uniformly convex hypersurface in $\R^{n}$.
Without loss of generality, we may assume that $M$ encloses the origin.
The support function $h$ of $M$ is defined as
\begin{equation*}
h(x) := \max_{y\in M} \langle y,x \rangle, \quad \forall x\in\uS,
\end{equation*}
where $\langle \cdot,\cdot \rangle$ is the standard inner product in $\R^n$.
And the radial function $\rho$ of $M$ is given by
\begin{equation*}
\rho(u) :=\max\set{\lambda>0 : \lambda u\in M}, \quad\forall u\in\uS.
\end{equation*}
Note that $\rho(u)u\in M$.

Denote the Gauss map of $M$  by $\nu_M$.
Then $M$ can be parametrized by the inverse Gauss map $X :
\mathbb{S}^{n-1}\rightarrow M$ with $X(x) =\nu_M^{-1}(x)$.
The support function $h$ of $M$ can be computed by
\begin{equation} \label{h}
  h(x) = \langle x, X(x)\rangle, \indent x \in \mathbb{S}^{n-1}. 
\end{equation}
Note that $x$ is just the unit outer normal of $M$ at $X(x)$.
Let $e_{ij}$ be the standard metric of the sphere $\mathbb{S}^{n-1}$, and
$\nabla$ be the corresponding connection on $\mathbb{S}^{n-1}$.
Differentiating \eqref{h}, we have
\begin{equation*}
  \nabla_{i} h = \langle \nabla_{i}x, X(x)\rangle + \langle x, \nabla_{i}X(x)\rangle. 
\end{equation*}
  Since $\nabla_{i}X(x)$ is tangent to $M$ at $X(x)$, we have
\begin{equation*}
  \nabla_{i} h = \langle \nabla_{i}x, X(x)\rangle. 
\end{equation*}
It follows that
\begin{equation}\label{Xh}
  X(x) = \nabla h + hx.
\end{equation}

By differentiating \eqref{h} twice, the  second fundamental form $A_{ij}$   of $M$
can be computed in terms of the support function, see for example \cite{Urb.JDG.33-1991.91},
\begin{equation}
\label{A} A_{ij} =  \nabla_{ij}h + he_{ij}, 
\end{equation}
where $\nabla_{ij} = \nabla_{i}\nabla_{j}$ denotes the second order covariant derivative with respect to $e_{ij}$.
The  induced metric matix $g_{ij}$ of $M$ can be derived by Weingarten's formula,
\begin{equation}
  \label{g}
  e_{ij} = \langle \nabla_{i}x, \nabla_{j}x\rangle  = A_{ik}A_{lj}g^{kl}. 
\end{equation}
The principal radii of curvature are the eigenvalues of the matrix $b_{ij} =
A^{ik}g_{jk}$.
When considering a smooth local orthonormal frame on $\uS$, by virtue of
\eqref{A} and \eqref{g}, we have
\begin{equation}
  \label{radii}
  b_{ij} = A_{ij} = \nabla_{ij}h + h\delta_{ij}.
\end{equation}
We will use
$b^{ij}$ to denote the inverse matrix of $b_{ij}$.
The Gauss curvature of $X(x) \in M$ is given by
\begin{equation*}
\mathcal{K}(x) = [\det(\nabla_{ij}h + h\delta_{ij})]^{-1}. 
\end{equation*}

From the evolution equation of $X(x,t)$ in flow \eqref{feq}, we derive the
evolution equation of the corresponding support function $h(x,t)$:
\begin{equation}\label{seq}
 \frac{\pd h}{\pd t} (x,t)= -\eta(t)f(x)\mathcal{K}h/\varphi(h)+ h(x,t), \indent x \in \mathbb{S}^{n-1}.
\end{equation}
Denote the radial function of $M_t$ by $\rho(u,t)$.
From $\eqref{Xh}$, $u$ and $x$ are related by
\begin{equation}
  \label{rs}
  \rho(u)u = \nabla h(x) + h(x)x.
\end{equation}
Let $x = x(u,t)$, by $\eqref{rs}$, we have
\begin{equation*}
\log \rho(u,t) = \log h(x,t) - \log \langle x,u \rangle.
\end{equation*}
Differentiating the above identity, we have
\begin{equation*}
  \begin{split}
    \frac{1}{\rho(u,t)}\frac{\pd \rho(u,t)}{\pd t}
    &= \frac{1}{h(x,t)}\Bigl(\nabla h\cdot \frac{\pd x(u,t)}{\pd t} + \frac{\pd h(x,t)}{\pd t}\Bigr)
    - \frac{u}{\langle x,u \rangle} \cdot \frac{\pd x(u,t)}{\pd t}\\
    &= \frac{1}{h(x,t)}\frac{\pd h(x,t)}{\pd t}
    + \frac{1}{h(x,t)}[\nabla h - \rho(u,t)u]\cdot \frac{\pd x(u,t)}{\pd t}\\
    &= \frac{1}{h(x,t)}\frac{\pd h(x,t)}{\pd t}.
  \end{split}
\end{equation*}
The evolution equation of  radial function then follows from \eqref{seq},
\begin{equation}\label{req}
\frac{\pd \rho}{\pd t} (u,t)= -\eta(t)f(x)\mathcal{K}\rho/\varphi(h) + \rho(u,t),
\end{equation}
where $\mathcal{K}$ denotes the Gauss curvature at $\rho(u,t)u \in M_{t}$ and
$f$ takes value at the unit normal vector $x(u,t)$.

We use $\VOL(t)$ to denote the volume of the convex body bounded by the
hypersurface $M_t$.
From
\begin{equation*}
  \VOL(t)=\frac{1}{n} \int_{\uS} \rho(u,t)^n \dd u,
\end{equation*}
we have the following computations:
\begin{equation*}
\begin{split}
  \pd_t \VOL(t)
  &= \int_{\uS} \rho(u,t)^{n-1}\pd_t\rho \dd u,\\
  &= \int_{\uS} \rho^n \dd u -\eta(t)\int_{\uS} \rho^nf(x)\mathcal{K}/\varphi(h) \dd u,\\
  &= \int_{\uS} \rho^n \dd u -\eta(t)\int_{\uS} f(x)h/\varphi(h) \dd x.
\end{split}
\end{equation*}
If we take $\eta(t)$ as in \eqref{eta}, there is
\begin{equation} \label{eq:1}
\pd_t \VOL(t)\equiv0, 
\end{equation}
namely the volume of the convex body bounded by $M_t$ remains unchanged.

\section{Long-time existence of the flow}

In this section, we will give a priori estimates about support functions and
obtain the long-time existence of flow \eqref{feq} under assumptions of Theorem \ref{thm1}.

In the following of this paper, we always assume that $M_{0}$ is a smooth,
closed, origin-symmetric, uniformly convex hypersurface in $\R^{n}$, $f$ is a
smooth, positive and even function on $\uS$, and $\varphi\in C^\infty(0,+\infty)$
satisfies {\bf (A)} or {\bf (B)}.
$h: \uS\times[0,T)\to \R$ is a smooth solution to the evolution equation \eqref{seq}
with the initial $h(\cdot,0)$ the support function of $M_0$.
Here $T$ is the maximal time for which the solution exists.
Let $M_t$ be the convex hypersurface determined by $h(\cdot,t)$, and
$\rho(\cdot,t)$ be the corresponding radial function.

We first prove the uniform positive upper and lower bounds of $h(\cdot,t)$ for $t\in[0,T)$.

\begin{lemma}\label{lem3.1}
When $\varphi$ satisfies {\bf (A)}. There exists a positive constant $C$
independent of $t$, such that for every $t\in[0,T)$
\begin{equation}\label{rho1}
1/C \leq \rho(\cdot,t) \leq C \text{ on }\uS.
\end{equation}
It means that
\begin{equation}\label{h1}
1/C \leq h(\cdot,t) \leq C \text{ on }\uS.
\end{equation}
\end{lemma}

\begin{proof}
Let
\begin{equation*}
J(t)=\int_{\uS} \phi(h(x,t)) f(x) \dd x, \quad t\geq 0.
\end{equation*}  
We claim that $J(t)$ is non-increasing.
In fact, recalling \eqref{seq}, we have
\begin{equation*}
  \begin{split}
    J'(t)
    &= \int_{\uS} \phi'(h(x,t)) \pd_th f(x) \dd x \\
    &= \int_{\uS} [-\eta(t)f(x)\mathcal{K}h/\varphi(h)+ h]  f(x) /\varphi(h) \dd x \\
    &= \int_{\uS}  f(x) h /\varphi(h) \dd x 
    -\eta(t) \int_{\uS} f(x)^2\,\mathcal{K}h / \varphi(h)^2 \dd x.
  \end{split}
\end{equation*}
By the definition of $\eta(t)$ in \eqref{eta}, there is 
\begin{equation*}
  \eta(t) = \frac{\int_{\uS} h/\mathcal{K} \dd x}{\int_{\uS} f(x)h/\varphi(h) \dd x}.
\end{equation*}
Hence
\begin{equation}\label{eq:2}
  \begin{split}
    J'&(t) \int_{\uS} f(x)h/\varphi(h) \dd x \\
    &= \left( \int_{\uS}  f(x) h /\varphi(h) \dd x  \right)^2
    - \int_{\uS} h/\mathcal{K} \dd x \cdot \int_{\uS} f(x)^2\,\mathcal{K}h / \varphi(h)^2 \dd x \\
    &= \left( \int_{\uS} \sqrt{h/\mathcal{K}} \cdot f \sqrt{\mathcal{K}h} /\varphi(h) \dd x  \right)^2
    - \int_{\uS} h/\mathcal{K} \dd x \cdot \int_{\uS} f^2\,\mathcal{K}h / \varphi(h)^2 \dd x \\
    &\leq 0,
  \end{split}
\end{equation}
where the last inequality is due to the H\"older's inequality.
Therefore, $J(t)$ is non-increasing.

For each $t$, write
\begin{equation*}
R_t = \max_{u\in\uS} \rho (u,t) = \rho(u_t,t)
\end{equation*}
for some $u_t\in\uS$.
Since $M_{t}$ is origin-symmetric, we have by the definition of support function
that
\begin{equation*}
h(x,t)\geq R_t|\langle x,u_t \rangle|, \quad \forall x\in\uS.
\end{equation*}

Now we have the following estimates:
\begin{equation*}
\begin{split}
  J(0) &\geq J(t) \\
  &\geq f_{\min} \int_{\uS} \phi(h(x,t)) \dd x \\
  &\geq f_{\min} \int_{\uS} \phi(R_t|\langle x,u_t \rangle|) \dd x \\
  &= f_{\min} \int_{\uS} \phi(R_t|x_1|) \dd x.
\end{split}
\end{equation*}
Denote $S_1 =\set{x\in\uS : |x_1|\geq 1/2}$, then
\begin{equation*}
\begin{split}
  J(0) 
  &\geq f_{\min} \int_{\uS} \phi(R_t/2) \dd x \\
  &= f_{\min} \phi(R_t/2) |S_1|,
\end{split}
\end{equation*}
which implies that $\phi(R_t/2)$ is uniformly bounded from above.
By assumption {\bf (A)}, $\phi(s)$ is strictly increasing and tends to $+\infty$
as $s\to+\infty$.
Thus $R_t$ is uniformly bounded from above.

Recalling $\VOL(t)\equiv\VOL(0)$ by \eqref{eq:1}, one can easily obtain the
uniform positive lower bound of $\rho(\cdot,t)$.
In fact, by the concept of minimum ellipsoid of a convex body, there exists a
positive constant $C_n$ depending only on $n$, such that
\begin{equation*}
\VOL(t) \leq C_n R_t^{n-1} \cdot\min_{u\in\uS} \rho(u,t).
\end{equation*}
Thus the uniform positive lower bound of $\rho(\cdot,t)$ follows from their uniform
upper bound.
\end{proof}

\begin{lemma}\label{lem3.2}
When $\varphi$ satisfies {\bf (B)}. There exists a positive constant $C$
independent of $t$, such that for every $t\in[0,T)$
\begin{equation}\label{h2}
1/C \leq h(\cdot,t) \leq C \text{ on }\uS.
\end{equation}
It means that
\begin{equation}\label{rho2}
1/C \leq \rho(\cdot,t) \leq C \text{ on }\uS.
\end{equation}
\end{lemma}

\begin{proof}
Let
\begin{equation*}
J(t)=\int_{\uS} \phi(h(x,t)) f(x) \dd x, \quad t\geq 0.
\end{equation*}  
Note that $\phi'(s)=-1/\varphi(s)$. We have
\begin{equation*}
  \begin{split}
    J'(t)
    &= \int_{\uS} \phi'(h(x,t)) \pd_th f(x) \dd x \\
    &= -\int_{\uS} [-\eta(t)f(x)\mathcal{K}h/\varphi(h)+ h]  f(x) /\varphi(h) \dd x \\
    &= \eta(t) \int_{\uS} f(x)^2\,\mathcal{K}h / \varphi(h)^2 \dd x
    -\int_{\uS}  f(x) h /\varphi(h) \dd x \\
    &\geq0,
  \end{split}
\end{equation*}
where the last inequality is due to \eqref{eq:2}.
Hence $J(t)$ is non-decreasing, and 
\begin{equation}\label{eq:3}
  J(t)\geq J(0)>0.
\end{equation}

Let $a$ be a positive number to be determined. Write
\begin{equation*}
S_t =\set{x\in\uS : h(x,t)\geq a}.
\end{equation*}
Then
\begin{equation*}
  \int_{S_t} \phi(h(x,t)) f(x) \dd x
  \leq \int_{S_t} \phi(a) f(x) \dd x
  \leq \phi(a) \norm{f}_{L^1(\uS)}.
\end{equation*}
By $\lim_{s\to+\infty}\phi(s)=0$, one can take a sufficiently large $a$, such
that
\begin{equation*}
  \int_{S_t} \phi(h(x,t)) f(x) \dd x
  <J(0)/2, \quad \forall\, t\in[0,T).
\end{equation*}
Note that $a$ depends only on $f$ and $\varphi$.

Recall assumption {\bf (B)}, when $s$ near $0$, $\phi(s)\leq N s^p$ for
some positive constant $N$ and some number $p\in(-n,0)$.
Since $\phi$ is smooth in $(0,+\infty)$, one can easily see that there exists a
positive number $\tilde{N}$ such that
\begin{equation*}
\phi(s)\leq \tilde{N} s^p, \quad \forall s\in(0,a).
\end{equation*}
Now we can estimate $J(t)$ as follows:
\begin{equation*}
  \begin{split}
    J(t)
    &=\biggl( \int_{\uS\backslash S_t} +\int_{S_t} \biggr) \phi(h(x,t)) f(x) \dd x \\
    &\leq \tilde{N} \int_{\uS\backslash S_t} h(x,t)^p f(x) \dd x +J(0)/2 \\
    &\leq \tilde{N} f_{\max} \int_{\uS} h(x,t)^p \dd x +J(0)/2,
  \end{split}
\end{equation*}
which together with \eqref{eq:3} implies that
\begin{equation*}
  \int_{\uS} h(x,t)^p \dd x
  \geq
  \frac{J(0)}{2\tilde{N} f_{\max}}.
\end{equation*}
Noting $h(\cdot,t)$ is even, $p\in(-n,0)$ and $\VOL(t)\equiv\VOL(0)$, by a
simple computation, one can see that $h(\cdot,t)$ has uniform positive upper and
lower bounds.
\end{proof}

Since $h(\cdot,t)$ is the support function, it is easy to obtain gradient
estimates from the bounds of $h(\cdot,t)$.
In fact, by the equality $\rho^{2} = h^{2} + |\nabla h|^{2}$, we have from the previous
lemmas that

\begin{corollary}\label{cor3.2}
  Under the assumptions of Theorem \ref{thm1}, we have
\begin{equation*}
  |\nabla h(x,t)| \leq C, \quad \forall (x,t) \in \mathbb{S}^{n-1} \times [0, T),
\end{equation*}
where $C$ is a positive constant depending only on constants in Lemmas \ref{lem3.1} and \ref{lem3.2}.
\end{corollary}

To obtain the long-time existence of the flow \eqref{feq} or \eqref{seq}, we
further need to establish uniform upper and lower bounds for the principal
curvature.

In the rest of this section, we take a local orthonormal frame $\{e_{1},
\cdots, e_{n-1}\}$ on $\mathbb{S}^{n-1} $ such that the standard metric on
$\mathbb{S}^{n-1} $ is $\{\delta_{ij}\}$. 
Double indices always mean to sum from $1$ to $n-1$.
For convenience, we also write
\begin{gather*}
\psi=1/\varphi, \\
F = \eta(t)f(x)\mathcal{K}(x)h \psi(h).
\end{gather*}
By Lemmas \ref{lem3.1} and \ref{lem3.2}, for any $t\in[0,T)$, $h(\cdot,t)$
 always ranges within a bounded interval $I'=[1/C,C]$, where $C$ is the
 constant in these two lemmas.

We first derive the upper bound for the Gaussian curvature.

\begin{lemma}\label{lem3.3}
  Under the assumptions of Theorem \ref{thm1}, we have 
\begin{equation*}
\mathcal{K}(x,t) \leq C, \quad \forall (x,t)\in \uS\times[0,T), 
\end{equation*}
where $C$ is a positive constant independent of $t$.
\end{lemma}

\begin{proof}
  Consider the following auxiliary function:
  \beqs Q(x,t)  = \frac{1}{h- \varepsilon _{0}}(F - h) = \frac{-h_{t}}{h- \varepsilon _{0}},\eeqs
where $ \varepsilon _{0}$ is a positive constant satisfying
\beqs \varepsilon_{0} < \min_{\mathbb{S}^{n-1} \times [0, T)} h(x,t).\eeqs
Recalling that $F =\eta(t) f(x)\mathcal{K}(x)h\psi(h)$ and that $h$ has uniform
positive upper and lower bounds, the upper bound of $\mathcal{K}(x,t) $  follows from that
of $ Q(x,t)$.
Hence we only need to derive the upper bound of $ Q(x,t)$.

First we   compute
the evolution equation of  $ Q(x,t)$. Note that
\begin{equation*}
\nabla_{i} Q
= \frac{F_{i} - h_{i}}{h- \varepsilon _{0}} - \frac{F-h}{(h- \varepsilon _{0})^{2}}h_{i}, 
\end{equation*}
and
\begin{equation*}
\begin{split}
 \nabla_{ij} Q
&= \frac{F_{ij} -h_{ij}  }{h- \varepsilon _{0}} - \frac{(F_{i} - h_{i}) h_{j}}{(h- \varepsilon _{0})^{2}} - \frac{(F_{j} - h_{j})h_{i} + (F-h)h_{ij}}{(h- \varepsilon _{0})^{2}} + 2\frac{(F-h)h_{i}h_{j}}{(h- \varepsilon _{0})^{3}}\\
& = \frac{F_{ij} -h_{ij}  }{h- \varepsilon _{0}} - \frac{(F-h)h_{ij}}{(h- \varepsilon _{0})^{2}} - \frac{\nabla_{i}Q h_{j}}{h- \varepsilon _{0}}
- \frac{\nabla_{j}Q h_{i}}{h- \varepsilon _{0}}.
\end{split}
\end{equation*}
There is also that
\begin{equation*}
\frac{\pd Q}{\pd t} = \frac{F_{t} - h_{t}}{h- \varepsilon _{0}}  + \frac{h_{t}^{2}}{(h- \varepsilon _{0})^{2}}= \frac{F_{t} }{h- \varepsilon _{0}} + Q + Q^{2}.
\end{equation*}
We obtain the
 evolution equation of $ Q(x,t)$:
\begin{equation*}
\begin{split}
\frac{\pd Q}{\pd t} - Fb^{ij} \nabla_{ij} Q
&=  \frac{1}{h- \varepsilon _{0}}(F_{t} -Fb^{ij}F_{ij}) + Q + Q^{2}  \\
&\hskip1.1em + (Q+1)\frac{Fb^{ij}h_{ij}}{h- \varepsilon _{0}} +  \frac{ \nabla_{i} Q F{b^{ij}h_{j}}}{h- \varepsilon _{0}} +  \frac{ \nabla_{j} Q F{b^{ij}h_{i}}}{h- \varepsilon _{0}}.
\end{split}
\end{equation*}

Now we need to compute the evolution equation of $F$.
From the fact
\begin{equation}
\label{dK}\frac{\pd \mathcal{K}}{\pd b_{ij}} = -\mathcal{K}b^{ij}, 
\end{equation}
we have
\begin{equation*}
\begin{split}
f(x)h\psi(h)\eta(t)\frac{\pd \mathcal{K} }{\pd t} 
&= -Fb^{ij}(h_{ij} + \delta_{ij}h)_{t}  \\
&= -Fb^{ij}(h _{t})_{ij} - Fb^{ij}\delta_{ij}h_{t} \\
&= -Fb^{ij}(-F + h)_{ij}  - Fb^{ij}\delta_{ij}h_{t} \\
&= Fb^{ij}F_{ij} - Fb^{ij}b_{ij} +  F^{2}b^{ij}\delta_{ij}.
\end{split}
\end{equation*}
Then there is
\begin{equation*}
\begin{split}
F_{t}
&= f(x)h\psi(h)\eta(t)\frac{\pd \mathcal{K} }{\pd t} +  \mathcal{K}(x,t)f(x)\frac{\pd }{\pd t}(h\psi(h)\eta(t))\\
&= Fb^{ij}F_{ij} - Fb^{ij}b_{ij} +  F^{2}b^{ij}\delta_{ij} + \mathcal{K}(x,t)f(x)\frac{\pd }{\pd t}(h\psi(h)\eta(t)).
\end{split}
\end{equation*}
Thus we obtain
\beqs \frac{\pd F}{\pd t}  -Fb^{ij}\nabla_{ij}F= - F(n-1) +  F^{2}b^{ij}\delta_{ij}+ \mathcal{K}(x,t)f(x)\frac{\pd }{\pd t}(h\psi(h)\eta(t)).\eeqs

At a spatial maximal point of $Q(x,t)$, if we take an orthonormal frame such that $b_{ij}$ is diagonal, we have
\begin{equation*}
\begin{split}
 \frac{\pd Q}{\pd t}  & -  b^{ii}F\nabla_{ii}Q \\
&\leq   \frac{1}{h- \varepsilon _{0}}(F_{t} - b^{ii}F\nabla_{ii}F) + Q + Q^{2}
+  \frac{Fb^{ii}h_{ii}}{h- \varepsilon _{0}} +  \frac{QF{b^{ii}h_{ii}}}{h- \varepsilon _{0}}\\
& =  \frac{1}{h- \varepsilon _{0}}[- Fb^{ii}b_{ii}
+ F^{2}b^{ii}\delta_{ii}+ \mathcal{K}(x,t)f(x)\frac{\pd }{\pd t}(h\psi(h)\eta(t))]\\
&\hskip1.1em  +  Q + Q^{2} +  \frac{Fb^{ii}(b_{ii} - h\delta_{ii})}{h- \varepsilon _{0}} +  \frac{QF{b^{ii}(b_{ii} - h\delta_{ii})}}{h- \varepsilon _{0}}\\
& =  \frac{F^{2}}{h- \varepsilon _{0}}\sum_{i}{b^{ii}}+ Q + Q^{2} + \frac{1}{h- \varepsilon _{0}}\mathcal{K}(x,t)f(x)\frac{\pd }{\pd t}(h\psi(h)\eta(t))\\
&\hskip1.1em - \frac{hF}{h- \varepsilon _{0}}\sum_{i}{b^{ii}} + \frac{QF(n-1)}{h- \varepsilon _{0}}  - \frac{QFh}{h- \varepsilon _{0}}\sum_{i}{b^{ii}}\\
& \leq  FQ\Bigl(1- \frac{h}{h- \varepsilon _{0}}\Bigr)\sum_{i}{b^{ii}} + C_{1}Q +  C_{2} Q^{2}  \\
&\hskip1.1em + \frac{1}{h- \varepsilon _{0}}\mathcal{K}(x,t)f(x)\frac{\pd }{\pd t}(h\psi(h)\eta(t)).
\end{split}
\end{equation*}

Since
\begin{equation*}
  \begin{split}
    \frac{\pd \eta(t)}{\pd t}
    &= - \frac{\int_{\uS}\rho^{n}\dd u }{[\int_{\uS}h\psi(h)f(x)\dd x]^2 }
    \int_{\uS} f(x)[\psi'(h)h + \psi(h)]h_{t}  \dd x \\
    &\leq C_3 Q,
  \end{split}
\end{equation*}
where $C_3$ is a positive constant depending on $\norm{f}_{C(\uS)}$, $\|\varphi\|_{C^1(I')}$ and
$\|h\|_{C^{1}(\mathbb{S}^{n-1} \times [0, T))}$.
Therefore we have
\begin{equation*}
\begin{split}
\frac{\pd }{\pd t}(h\psi(h)\eta(t))
&=h\psi(h) \frac{\pd \eta(t)}{\pd t} + [\psi'(h)h + \psi(h)]\eta(t)h_{t}\\
&\leq  C_4 Q.
\end{split}
\end{equation*}
For $Q$ large enough, there is
$$ 1/C_0 \, \mathcal{K}\leq Q \leq C_{0}\mathcal{K},$$
and
\beqs \sum_{i}{b^{ii}} \geq (n-1) \mathcal{K}^{\frac{1}{n-1}}.\eeqs
Hence, for large $Q$, we obtain
\beqs \frac{\pd Q}{\pd t}\leq C_{1}Q^{2} (C_{2} - \varepsilon _{0}Q^{\frac{1}{n-1}}) < 0.\eeqs
Then the upper bound of $\mathcal{K}(x,t)$ follows.
\end{proof}

Now we can estimate lower bounds of principal curvatures $\kappa_{i}(x,t)$ of $M_t$
for $i=1,\cdots, n-1$.

\begin{lemma}\label{lem3.4}
  Under the assumptions of Theorem \ref{thm1}, we have
  \begin{equation*}
  \kappa_{i} \geq C, \quad \forall (x,t)\in\uS\times[0,T),
  \end{equation*}
  where $C$ is a positive constant independent of $t$.
\end{lemma}

\begin{proof}
Consider the auxiliary function
$$w (x,t) = \log \lambda_{\max}(b_{ij}) - A\log h + B|\nabla h|^{2},$$
where $A,B$ are constants to be determined. $\lambda_{\max} $ is the maximal eigenvalue of $b_{ij}$.

For any fixed $t$, we assume that $\max\limits_{\mathbb{S}^{n-1}}w (x,t)$ is
attained at $q \in \mathbb{S}^{n-1}$. At $q $, we take an orthogonal frame such that
$b_{ij}(q,t)$ is diagonal and $\lambda_{\max} (q,t) = b_{11}(q,t)$.
Now we can write $w (x,t)$ as
$$w (x,t) = \log b_{11} - A\log h + B|\nabla h|^{2}.$$

We first compute the evolution equation of $w$. Note that
\begin{align*}
  \frac{\pd \log{b_{11}}}{\pd t} - b^{ii}F\nabla_{ii}\log{b_{11}}
  &= b^{11}\Bigl(\frac{\pd b_{11}}{\pd t}  - Fb^{ii}\nabla_{ii}b_{11} \Bigr) + Fb^{ii}(b^{11})^{2}(\nabla_{i}b_{11})^{2}, \\
  \frac{\pd \log{h}}{\pd t} - b^{ii}F\nabla_{ii}\log{h}
  &= \frac{1}{h}\Bigl(\frac{\pd h}{\pd t}  - Fb^{ii}\nabla_{ii}h \Bigr) +  \frac{Fb^{ii}h_{i}^{2}}{h^{2}}, \\
  \frac{\pd |\nabla h|^{2}}{\pd t} - b^{ii}F\nabla_{ii}|\nabla h|^{2}
  &= 2h_{k}\Bigl(\frac{\pd h_{k}}{\pd t}  - Fb^{ii}\nabla_{ii}h_{k} \Bigr) - 2Fb^{ii}h_{ii}^{2},
\end{align*}
we have
\begin{equation}\label{eq:4}
  \begin{split}
    \frac{\pd w}{\pd t}  -  Fb^{ii}\nabla_{ii} w
    &= b^{11}\Bigl(\frac{\pd b_{11}}{\pd t}  - Fb^{ii}\nabla_{ii}b_{11} \Bigr)  + Fb^{ii}(b^{11})^{2}(\nabla_{i}b_{11})^{2} \\
    &\hskip1.1em - \frac{A}{h}\Bigl(\frac{\pd h}{\pd t}
    - Fb^{ii}\nabla_{ii}h \Bigr) -   \frac{Fb^{ii}h_{i}^{2}}{h^{2}} \\
    &\hskip1.1em + 2Bh_{k}\Bigl(\frac{\pd h_{k}}{\pd t}  - Fb^{ii}\nabla_{ii}h_{k} \Bigr) - 2BFb^{ii}h_{ii}^{2}.
  \end{split}
\end{equation}

Let
\begin{equation*}
M = \log [hf(x)\psi(h)\eta(t)],
\end{equation*}
then $$\log F =  \log \mathcal{K} + M.$$
Differentiating the above equation, we have by \eqref{dK} that
\begin{equation*}
\frac{\nabla_{k}F}{F}
= \frac{1}{\mathcal{K}}\frac{\pd \mathcal{K}}{\pd b_{ij}} \nabla_{k}b_{ij} +  \nabla_{k}M
= -b^{ij}\nabla_{k}b_{ij} +  M_{k},
\end{equation*}
and
\begin{equation*}
\frac{\nabla_{kl}F}{F}  - \frac{\nabla_{k}F\nabla_{l}F}{F^{2}}
= -b^{ij}\nabla_{kl}b_{ij} + b^{ii}b^{jj}\nabla_{k}b_{ij}\nabla_{l}b_{ij} +  \nabla_{lk}M.
\end{equation*}
Recalling the evolution equation of $h$, we have
\begin{equation}\label{eq:7}
\begin{split}
\frac{\pd h}{\pd t}  - Fb^{ii}\nabla_{ii}h &= -F + h - Fb^{ii}(b_{ii} - \delta_{ii}h)\\
&=-Fn + h + Fh \sum_{i}b^{ii},
\end{split}
\end{equation}
and
\begin{equation}\label{eq:8}
\begin{split}
\frac{\pd h_{k}}{\pd t} - Fb^{ii}\nabla_{ii}h_{k}
&= -F_{k} + h_{k} - Fb^{ii} \nabla_{k}b_{ii} + Fh_{k}\sum_{i}b^{ii}\\
&= -M_{k}F + h_{k} + Fh_{k}\sum_{i}b^{ii}.
\end{split}
\end{equation}
We also have
\begin{equation*}
  \begin{split}
    \frac{\pd h_{kl}}{\pd t}
    & = - \nabla_{kl}F + h_{kl} \\
    &= - \frac{\nabla_{k}F\nabla_{l}F}{F}  + Fb^{ij}\nabla_{kl}b_{ij} - Fb^{ii}b^{jj}\nabla_{k}b_{ij}\nabla_{l}b_{ij} - F \nabla_{lk}M +h_{kl}.
  \end{split}
\end{equation*}
By the Gauss equation, see e.g. \cite{Urb.JDG.33-1991.91} for details,
 \beqs\nabla_{kl} h_{ij} = \nabla_{ij} h_{kl} + 2\delta_{kl}h_{ij} - 2\delta_{ij}h_{kl} + \delta_{kj}h_{il} - \delta_{li}h_{kj},\eeqs
or \beqs \nabla_{kl}b_{ij } = \nabla_{ij} b_{kl} + \delta_{kl}h_{ij} - \delta_{ij}h_{kl} + \delta_{kj}h_{il} - \delta_{li}h_{kj}.\eeqs
Then
\begin{equation*}
  \begin{split}
    \frac{\pd h_{kl}}{\pd t}
&= Fb^{ij}\nabla_{ij}h_{kl} + 2\delta_{kl}Fb^{ij}h_{ij} -  Fb^{ij}\delta_{ij}h_{kl} + Fb^{ik}h_{il} - Fb^{jl}h_{kj}\\
&\hskip1.1em - \frac{\nabla_{k}F\nabla_{l}F}{F}    - Fb^{ii}b^{jj}\nabla_{k}b_{ij}\nabla_{l}b_{ij} - F \nabla_{lk}M +h_{kl}.
\end{split}
\end{equation*}
Hence
\begin{equation*}
\begin{split} \frac{\pd b_{kl}}{\pd t}
&= Fb^{ij}\nabla_{ij}b_{kl} + \delta_{kl}Fb^{ij}(b_{ij} -h\delta_{ij} )-  Fb^{ij}\delta_{ij}(b_{kl}- h\delta_{kl})\\
&\hskip1.1em + Fb^{ik}h_{il} - Fb^{jl}h_{kj} + b_{kl} - h\delta_{kl} + (-F + h)\delta_{kl}   \\
&\hskip1.1em- \frac{\nabla_{k}F\nabla_{l}F}{F}- Fb^{ii}b^{jj}\nabla_{k}b_{ij}\nabla_{l}b_{ij} - F \nabla_{lk}M \\
&= Fb^{ij}\nabla_{ij}b_{kl} + \delta_{kl}F(n-2) - Fb^{ij}\delta_{ij}b_{kl}
+ Fb^{ik}h_{il} - Fb^{jl}h_{kj} + b_{kl} \\
&\hskip1.1em- \frac{\nabla_{k}F\nabla_{l}F}{F}   - Fb^{ii}b^{jj}\nabla_{k}b_{ij}\nabla_{l}b_{ij} - F \nabla_{lk}M.
\end{split}
\end{equation*}
When $k=l=1$, we  have
\begin{equation}\label{eq:6}
  \begin{split}
    \frac{\pd b_{11}}{\pd t}
&= Fb^{ii}\nabla_{ii}(b_{11}) + F(n-2) - F\sum_{i}b^{ii}b_{11} + b_{11}\\
&\hskip1.1em - \frac{F_{1}^{2}}{F}   - Fb^{ii}b^{jj}\nabla_{1}(b_{ij})^{2} - F M_{11} .
\end{split}
\end{equation}

Inserting \eqref{eq:7}, \eqref{eq:8} and \eqref{eq:6} into \eqref{eq:4}, we obtain
\begin{equation*}
  \begin{split}
    \frac{\pd w}{\pd t}
    & -  Fb^{ii}\nabla_{ii} w \\
    &\leq b^{11}F(n-2) + 1 - A - b^{11}F M_{11} + \frac{AFn}{h} - AF\sum_{i}b^{ii}  -2BFh_{k}M_{k}   \\
    &\hskip1.1em + 2B|\nabla h|^{2} + 2BF|\nabla h|^{2}\sum_{i}b^{ii} - 2BF \sum_{i}b_{ii} + 4BFh(n-1)\\
    &= b^{11}F(n-2) - b^{11}F M_{11}-2BFh_{k}M_{k} - (A-2B|\nabla h|^{2})F\sum_{i}b^{ii}\\
    &\hskip1.1em -(A -1- 2B |\nabla h|^{2})  - 2BF \sum_{i}b_{ii} + 4BFh(n-1)+ \frac{AFn}{h}.
  \end{split}
\end{equation*}
If we let $A$ satisfy $$ A \geq 2B \max_{\mathbb{S}^{n-1} \times [0, T)}{|\nabla h|^{2}} + 1,$$
then
\begin{equation}\label{eq:5}
  \begin{split}
    \frac{\pd w}{\pd t}  -  Fb^{ii}\nabla_{ii} w
    &\leq b^{11}F(n-2) - b^{11}F M_{11}-2BFh_{k}M_{k} \\
    &\hskip1.1em - 2BF \sum_{i}b_{ii} + 4BFh(n-1)+ \frac{AFn}{h}.
  \end{split}
\end{equation}

Now, we estimate $  - b^{11}F M_{11}-2BFh_{k}M_{k}. $
Since
\begin{equation*}
  \begin{split}
    M &= \log [hf(x)\psi(h)\eta(t)] \\
    & = \log {f(x)} + \log {h(x)} + \log{\psi(h)} + \log {\eta(t)},
  \end{split}
\end{equation*}
then
\begin{equation*}
    \nabla_{k}M
    = \frac{f_{k}}{f} + \frac{h_{k}}{h} + \frac{\psi'}{\psi}h_{k},
\end{equation*}
and
\begin{equation*}
    \nabla_{11}M
    = \frac{f_{11}}{f} - \frac{f_{1}^{2}}{f^{2}} + \frac{h_{11}}{h} - \frac{h_{1}^{2}}{h^{2}}
    +  \frac{\psi''h_{1}^{2} + \psi'h_{11}}{\psi} - \frac{(\psi')^{2}h_{1}^{2}}{\psi^{2}}.
\end{equation*}
Therefore, we obtain
\begin{equation*}
  \begin{split}
    -2Bh_{k}M_{k}
&= -2Bh_{k}\Bigl(\frac{f_{k}}{f} + \frac{h_{k}}{h} + \frac{\psi'}{\psi}h_{k}\Bigr)\\
&\leq 2B\Bigl(\frac{|\nabla f||\nabla h|}{f} +\frac{|\nabla h|^2}{h} + |\nabla h|^{2}\frac{|\psi'|}{\psi}  \Bigr) \\
&\leq  c_{1}B,
\end{split}
\end{equation*}
where $c_{1}$ is a positive constant depending on upper and lower bounds of $f$,
$\varphi(h)$ and $h$, and upper bounds of their first order derivatives.
We also have
\begin{equation*}
  \begin{split}
    - b^{11} M_{11}
    &=  - b^{11} \Bigl(\frac{f_{11}}{f} - \frac{f_{1}^{2}}{f^{2}}  - \frac{h_{1}^{2}}{h^{2}}  +  \frac{\psi''h_{1}^{2} }{\psi} - \frac{(\psi')^{2}h_{1}^{2}}{\psi^{2}} \Bigr)
    + b^{11}\frac{\psi'(b_{11} - h)}{\psi}- b^{11}\frac{b_{11}-h}{h}\\
&\leq  c_{2}b^{11} + c_{3},
\end{split}
\end{equation*}
where  $c_{2}$, $c_3$ are positive constants depending  on $\|\varphi\|_{C^2(I')}$,
$\|f\|_{C^{2}(\uS)}$, $\|h\|_{C^1(\mathbb{S}^{n-1} \times [0, T))}$, and
lower bounds of $\varphi(h)$, $f$ and $h$.
Thus, we have proved that
\begin{equation*}
 - b^{11}F M_{11}-2BFh_{k}M_{k}
\leq  F(c_{1}B  + c_{2}b^{11} + c_{3} ).
\end{equation*}

Now, from \eqref{eq:5} we have
\begin{equation*}
\frac{\pd w}{\pd t}  -  Fb^{ii}\nabla_{ii} w
\leq
F(c_{1}B  + c_{2}b^{11} + c_{3} )  - 2BF \sum_{i}b_{ii} + 4BFh(n-1)+ \frac{AFn}{h}.
\end{equation*}
If we take $B=1$,
then for $b_{ii}$ large enough, there is
\begin{equation*}
  \begin{split}
    \frac{\pd w}{\pd t}  -  Fb^{ii}\nabla_{ii} w
    &\leq F(c_{1}  + c_{2}b^{11} + c_{3}  )  -2F \sum_{i}b_{ii} + 4Fh(n-1)+ \frac{AFn}{h} \\
    &< 0,
  \end{split}
\end{equation*}
which implies that
\beqs \frac{\pd w}{\pd t}  < 0. \eeqs
Therefore $w$ has a uniform upper bound, and so does $\lambda_{\max}(b_{ij})$.
The conclusion of this lemma then follows.
\end{proof}

Combining Lemma \ref{lem3.3} and Lemma \ref{lem3.4}, we see that
the principal curvatures of $M_{t}$ has uniform positive upper and lower bounds.
This together with Lemmas \ref{lem3.1} and \ref{lem3.2} implies that the
evolution equation \eqref{seq} is uniformly parabolic on any finite time
interval. Thus, the result of \cite{KS.IANSSM.44-1980.161} and the standard
parabolic theory show that the smooth solution of \eqref{seq} exists for all
time.
And by these estimates again, a subsequence of $M_t$ converges in $C^\infty$ to
a positive, smooth, uniformly convex hypersurface $M_\infty$ in $\R^n$.
Now to complete the proof of Theorem \ref{thm1}, it remains to check the support
function of $M_\infty$ satisfies Eq. \eqref{OMP-f}.

\section{Convergence of the flow}

In this section, we will complete the proof of Theorem \ref{thm1}.
Let $\tilde{h}$ be the support function of $M_\infty$. We need to prove that
$\tilde{h}$ is a solution to the following equation
\begin{equation} \label{OMP-f1}
  c\, \varphi(h) \det(\nabla^2h +hI) =f \text{ on } \uS
\end{equation}
for some positive constant $c$.

As before, we define the functional
\begin{equation*}
J(t)=\int_{\uS} \phi(h(x,t)) f(x) \dd x, \quad t\geq 0.
\end{equation*}  
By the assumptions on $\phi$, and Lemmas \ref{lem3.1} and \ref{lem3.2}, there
exists a positive constant $C$ which is independent of $t$, such that 
\begin{equation} \label{eq4.1}
  J(t) \leq C, \quad \forall t\geq0.
\end{equation}
We also note that, by proofs of these two lemmas, $J(t)$ is non-increasing when
$\varphi$ satisfies {\bf (A)}, and non-decreasing when $\varphi$ satisfies {\bf (B)}.

We now begin the proof with the assumption {\bf (A)}.
Recalling $J'(t)\leq0$ for any $t>0$.
From
\begin{equation*}
\int_0^t [-J'(t)] \dd t =J(0)-J(t) \leq J(0),
\end{equation*}
we have
\begin{equation*}
\int_0^\infty [-J'(t)] \dd t  \leq J(0),
\end{equation*}
This implies that there exists a subsequence of times $t_j\to\infty$ such that
\begin{equation*}
-J'(t_j) \to 0 \text{ as } t_j\to\infty.
\end{equation*}
Recalling \eqref{eq:2}:
\begin{equation*}
  \begin{split}
    J'&(t_j) \int_{\uS} f(x)h/\varphi(h) \dd x \\
    &= \left( \int_{\uS} \sqrt{h/\mathcal{K}} \cdot f \sqrt{\mathcal{K}h} /\varphi(h) \dd x  \right)^2
    - \int_{\uS} h/\mathcal{K} \dd x \cdot \int_{\uS} f^2\,\mathcal{K}h / \varphi(h)^2 \dd x.
  \end{split}
\end{equation*}
Since $h$ and $\mathcal{K}$ have uniform positive upper and lower bounds, by
passing to the limit, we obtain
\begin{equation*}
  \left( \int_{\uS} \sqrt{\tilde{h}/\widetilde{\mathcal{K}}} \cdot f \sqrt{\widetilde{\mathcal{K}}\tilde{h}} /\varphi(\tilde{h}) \dd x  \right)^2
  = \int_{\uS} \tilde{h}/\widetilde{\mathcal{K}} \dd x \cdot \int_{\uS} f^2\widetilde{\mathcal{K}}\tilde{h} / \varphi(\tilde{h})^2 \dd x,
  \end{equation*}
where $\widetilde{\mathcal{K}}$ is the Gauss curvature of $M_\infty$.
By the equality condition for the H\"older's inequality, there exists a constant
$c\geq0$ such that
\begin{equation*}
  c^2\, \tilde{h}/\widetilde{\mathcal{K}}
  = f^2\widetilde{\mathcal{K}}\tilde{h} / \varphi(\tilde{h})^2 \text{ on }\uS,
\end{equation*}
namely
\begin{equation*}
  c\, \varphi(\tilde{h})/\widetilde{\mathcal{K}} =f \text{ on }\uS,
\end{equation*}
which is just equation \eqref{OMP-f1}.
Note $\tilde{h}$ and $\widetilde{\mathcal{K}}$ have positive upper and lower
bounds, $c$ should be positive.

For the proof with the assumption {\bf (B)}.
Recalling $J'(t)\geq0$ for any $t>0$.
By estimate \eqref{eq4.1}, 
\begin{equation*}
\int_0^t J'(t) \dd t =J(t)-J(0) \leq J(t) \leq C,
\end{equation*}
which leads to
\begin{equation*}
\int_0^\infty J'(t) \dd t  \leq C.
\end{equation*}
This implies that there exists a subsequence of times $t_j\to\infty$ such that
\begin{equation*}
J'(t_j) \to 0 \text{ as } t_j\to\infty.
\end{equation*}
Now using almost the same arguments as above, one can prove $\tilde{h}$ solves
Eq. \eqref{OMP-f1} for some positive constant $c$.
The proof of Theorem \ref{thm1} is completed.


\begin{thebibliography}{10}

\bibitem{And.IMRN.1997.1001}
{\sc B.~Andrews}, {\em Monotone quantities and unique limits for evolving
  convex hypersurfaces}, Internat. Math. Res. Notices,  (1997), pp.~1001--1031.

\bibitem{And.Invent.138-1999.151}
\leavevmode\vrule height 2pt depth -1.6pt width 23pt, {\em Gauss curvature
  flow: the fate of the rolling stones}, Invent. Math., 138 (1999),
  pp.~151--161.

\bibitem{AGN.Adv.299-2016.174}
{\sc B.~Andrews, P.~Guan, and L.~Ni}, {\em Flow by powers of the {G}auss
  curvature}, Adv. Math., 299 (2016), pp.~174--201.

\bibitem{BBC.AiAM.111-2019.101937}
{\sc G.~Bianchi, K.~J. B\"{o}r\"{o}czky, and A.~Colesanti}, {\em The {O}rlicz
  version of the {$L_p$} {M}inkowski problem for {$-n< p<0$}}, Adv. in Appl.
  Math., 111 (2019), p.~101937.

\bibitem{BHZ.IMRNI.2016.1807}
{\sc K.~J. B\"or\"oczky, P.~Heged{\H u}s, and G.~Zhu}, {\em On the discrete
  logarithmic {M}inkowski problem}, Int. Math. Res. Not. IMRN,  (2016),
  pp.~1807--1838.

\bibitem{BLYZ.JAMS.26-2013.831}
{\sc K.~J. B{\"o}r{\"o}czky, E.~Lutwak, D.~Yang, and G.~Zhang}, {\em The
  logarithmic {M}inkowski problem}, J. Amer. Math. Soc., 26 (2013),
  pp.~831--852.

\bibitem{BCD.Acta.219-2017.1}
{\sc S.~Brendle, K.~Choi, and P.~Daskalopoulos}, {\em Asymptotic behavior of
  flows by powers of the {G}aussian curvature}, Acta Math., 219 (2017),
  pp.~1--16.

\bibitem{BIS.AP.12-2019.259}
{\sc P.~Bryan, M.~N. Ivaki, and J.~Scheuer}, {\em A unified flow approach to
  smooth, even {$L_p$}-{M}inkowski problems}, Anal. PDE, 12 (2019),
  pp.~259--280.

\bibitem{CHZ.MA.373-2019.953}
{\sc C.~Chen, Y.~Huang, and Y.~Zhao}, {\em Smooth solutions to the {$L_p$} dual
  {M}inkowski problem}, Math. Ann., 373 (2019), pp.~953--976.

\bibitem{CL}
{\sc H.~Chen and Q.-R. Li}, {\em The {$L_p$} dual {M}inkowski problem and
  related parabolic flows}.
\newblock Preprint.

\bibitem{CW.AIHPANL.17-2000.733}
{\sc K.-S. Chou and X.-J. Wang}, {\em A logarithmic {G}auss curvature flow and
  the {M}inkowski problem}, Ann. Inst. H. Poincar\'e Anal. Non Lin\'eaire, 17
  (2000), pp.~733--751.

\bibitem{CW.Adv.205-2006.33}
\leavevmode\vrule height 2pt depth -1.6pt width 23pt, {\em The
  {$L_p$}-{M}inkowski problem and the {M}inkowski problem in centroaffine
  geometry}, Adv. Math., 205 (2006), pp.~33--83.

\bibitem{Cho.JDG.22-1985.117}
{\sc B.~Chow}, {\em Deforming convex hypersurfaces by the {$n$}th root of the
  {G}aussian curvature}, J. Differential Geom., 22 (1985), pp.~117--138.

\bibitem{Fir.M.21-1974.1}
{\sc W.~J. Firey}, {\em Shapes of worn stones}, Mathematika, 21 (1974),
  pp.~1--11.

\bibitem{GL.DMJ.75-1994.79}
{\sc M.~E. Gage and Y.~Li}, {\em Evolving plane curves by curvature in relative
  geometries. {II}}, Duke Math. J., 75 (1994), pp.~79--98.

\bibitem{GHW.JDG.97-2014.427}
{\sc R.~J. Gardner, D.~Hug, and W.~Weil}, {\em The {O}rlicz-{B}runn-{M}inkowski
  theory: a general framework, additions, and inequalities}, J. Differential
  Geom., 97 (2014), pp.~427--476.

\bibitem{GHWY.JMAA.430-2015.810}
{\sc R.~J. Gardner, D.~Hug, W.~Weil, and D.~Ye}, {\em The dual
  {O}rlicz-{B}runn-{M}inkowski theory}, J. Math. Anal. Appl., 430 (2015),
  pp.~810--829.

\bibitem{Ger.CVPDE.49-2014.471}
{\sc C.~Gerhardt}, {\em Non-scale-invariant inverse curvature flows in
  {E}uclidean space}, Calc. Var. Partial Differential Equations, 49 (2014),
  pp.~471--489.

\bibitem{HLYZ.Adv.224-2010.2485}
{\sc C.~Haberl, E.~Lutwak, D.~Yang, and G.~Zhang}, {\em The even {O}rlicz
  {M}inkowski problem}, Adv. Math., 224 (2010), pp.~2485--2510.

\bibitem{HSX.MA.352-2012.517}
{\sc C.~Haberl, F.~E. Schuster, and J.~Xiao}, {\em An asymmetric affine
  {P}\'olya-{S}zeg\"o principle}, Math. Ann., 352 (2012), pp.~517--542.

\bibitem{Ham.CAG.2-1994.155}
{\sc R.~S. Hamilton}, {\em Remarks on the entropy and {H}arnack estimates for
  the {G}auss curvature flow}, Comm. Anal. Geom., 2 (1994), pp.~155--165.

\bibitem{HLW.CVPDE.55-2016.117}
{\sc Y.~He, Q.-R. Li, and X.-J. Wang}, {\em Multiple solutions of the
  {$L_p$}-{M}inkowski problem}, Calc. Var. Partial Differential Equations, 55
  (2016), pp.~Art. 117, 13 pp.

\bibitem{HP.Adv.323-2018.114}
{\sc M.~Henk and H.~Pollehn}, {\em Necessary subspace concentration conditions
  for the even dual {M}inkowski problem}, Adv. Math., 323 (2018), pp.~114--141.

\bibitem{HH.DCG.48-2012.281}
{\sc Q.~Huang and B.~He}, {\em On the {O}rlicz {M}inkowski problem for
  polytopes}, Discrete Comput. Geom., 48 (2012), pp.~281--297.

\bibitem{HLX.Adv.281-2015.906}
{\sc Y.~Huang, J.~Liu, and L.~Xu}, {\em On the uniqueness of
  {$L_p$}-{M}inkowski problems: the constant {$p$}-curvature case in
  {$\Bbb{R}^3$}}, Adv. Math., 281 (2015), pp.~906--927.

\bibitem{HLYZ.Acta.216-2016.325}
{\sc Y.~Huang, E.~Lutwak, D.~Yang, and G.~Zhang}, {\em Geometric measures in
  the dual {B}runn-{M}inkowski theory and their associated {M}inkowski
  problems}, Acta Math., 216 (2016), pp.~325--388.

\bibitem{HZ.Adv.332-2018.57}
{\sc Y.~Huang and Y.~Zhao}, {\em On the {$L_p$} dual {M}inkowski problem}, Adv.
  Math., 332 (2018), pp.~57--84.

\bibitem{HLYZ.DCG.33-2005.699}
{\sc D.~Hug, E.~Lutwak, D.~Yang, and G.~Zhang}, {\em On the {$L_p$} {M}inkowski
  problem for polytopes}, Discrete Comput. Geom., 33 (2005), pp.~699--715.

\bibitem{Iva.JFA.271-2016.2133}
{\sc M.~N. Ivaki}, {\em Deforming a hypersurface by {G}auss curvature and
  support function}, J. Funct. Anal., 271 (2016), pp.~2133--2165.

\bibitem{JL.Adv.344-2019.262}
{\sc H.~Jian and J.~Lu}, {\em Existence of solutions to the
  {O}rlicz-{M}inkowski problem}, Adv. Math., 344 (2019), pp.~262--288.

\bibitem{JLW.Adv.281-2015.845}
{\sc H.~Jian, J.~Lu, and X.-J. Wang}, {\em Nonuniqueness of solutions to the
  {$L_p$}-{M}inkowski problem}, Adv. Math., 281 (2015), pp.~845--856.

\bibitem{JLZ.CVPDE.55-2016.41}
{\sc H.~Jian, J.~Lu, and G.~Zhu}, {\em Mirror symmetric solutions to the
  centro-affine {M}inkowski problem}, Calc. Var. Partial Differential
  Equations, 55 (2016), pp.~Art. 41, 22 pp.

\bibitem{KS.IANSSM.44-1980.161}
{\sc N.~V. Krylov and M.~V. Safonov}, {\em A property of the solutions of
  parabolic equations with measurable coefficients}, Izv. Akad. Nauk SSSR Ser.
  Mat., 44 (1980), pp.~161--175, 239.

\bibitem{LSW.JEMSJ.22-2020.893}
{\sc Q.-R. Li, W.~Sheng, and X.-J. Wang}, {\em Flow by {G}auss curvature to the
  {A}leksandrov and dual {M}inkowski problems}, J. Eur. Math. Soc. (JEMS), 22
  (2020), pp.~893--923.

\bibitem{LL}
{\sc Y.~Liu and J.~Lu}, {\em A flow method for the dual {O}rlicz-{M}inkowski
  problem}.
\newblock Accepted by Trans. Amer. Math. Soc., arXiv:2001.08862.

\bibitem{Lu.SCM.61-2018.511}
{\sc J.~Lu}, {\em Nonexistence of maximizers for the functional of the
  centroaffine {M}inkowski problem}, Sci. China Math., 61 (2018), pp.~511--516.

\bibitem{LW.JDE.254-2013.983}
{\sc J.~Lu and X.-J. Wang}, {\em Rotationally symmetric solutions to the
  {$L_p$}-{M}inkowski problem}, J. Differential Equations, 254 (2013),
  pp.~983--1005.

\bibitem{Lud.Adv.224-2010.2346}
{\sc M.~Ludwig}, {\em General affine surface areas}, Adv. Math., 224 (2010),
  pp.~2346--2360.

\bibitem{Lut.JDG.38-1993.131}
{\sc E.~Lutwak}, {\em The {B}runn-{M}inkowski-{F}irey theory. {I}. {M}ixed
  volumes and the {M}inkowski problem}, J. Differential Geom., 38 (1993),
  pp.~131--150.

\bibitem{Schneider.2014}
{\sc R.~Schneider}, {\em Convex bodies: the {B}runn-{M}inkowski theory},
  vol.~151 of Encyclopedia of Mathematics and its Applications, Cambridge
  University Press, Cambridge, expanded~ed., 2014.

\bibitem{Sch.JRAM.600-2006.117}
{\sc O.~C. Schn\"{u}rer}, {\em Surfaces expanding by the inverse {G}au\ss
  curvature flow}, J. Reine Angew. Math., 600 (2006), pp.~117--134.

\bibitem{SY}
{\sc W.~Sheng and C.~Yi}, {\em An anisotropic shrinking flow and {$L_p$}
  {M}inkowski problem}.
\newblock arXiv:1905.04679.

\bibitem{Sta.IMRNI.2012.2289}
{\sc A.~Stancu}, {\em Centro-affine invariants for smooth convex bodies}, Int.
  Math. Res. Not. IMRN,  (2012), pp.~2289--2320.

\bibitem{SL.Adv.281-2015.1364}
{\sc Y.~Sun and Y.~Long}, {\em The planar {O}rlicz {M}inkowski problem in the
  {$L^1$}-sense}, Adv. Math., 281 (2015), pp.~1364--1383.

\bibitem{Urb.JDG.33-1991.91}
{\sc J.~Urbas}, {\em An expansion of convex hypersurfaces}, J. Differential
  Geom., 33 (1991), pp.~91--125.

\bibitem{XJL.Adv.260-2014.350}
{\sc D.~Xi, H.~Jin, and G.~Leng}, {\em The {O}rlicz {B}runn-{M}inkowski
  inequality}, Adv. Math., 260 (2014), pp.~350--374.

\bibitem{Zhu.Adv.262-2014.909}
{\sc G.~Zhu}, {\em The logarithmic {M}inkowski problem for polytopes}, Adv.
  Math., 262 (2014), pp.~909--931.

\bibitem{Zhu.IUMJ.66-2017.1333}
\leavevmode\vrule height 2pt depth -1.6pt width 23pt, {\em The {$L_p$}
  {M}inkowski problem for polytopes for {$p<0$}}, Indiana Univ. Math. J., 66
  (2017), pp.~1333--1350.

\bibitem{ZX.Adv.265-2014.132}
{\sc D.~Zou and G.~Xiong}, {\em Orlicz-{J}ohn ellipsoids}, Adv. Math., 265
  (2014), pp.~132--168.

\end{thebibliography}

\end{document}